\documentclass[12pt]{amsart} %

\usepackage[utf8]{inputenc}
\usepackage[T1]{fontenc}

\usepackage{geometry} 
\usepackage{graphicx} 

\newtheorem{theorem}{Theorem}[section]
\newtheorem{prop}[theorem]{Proposition}
\newtheorem{lemma}[theorem]{Lemma}
\newtheorem{cor}[theorem]{Corollary}

\theoremstyle{definition}
\newtheorem{defn}[theorem]{Definition}
\newtheorem{remark}[theorem]{Remark}

\def\Bbb{\mathbb} \def\cal{\mathcal}
\def\wt#1{\widetilde{#1}}
\def\wh#1{\widehat{#1}}
\def\wl#1{\overline{#1}}
\def\sier#1{{\cal O}_{#1}}
\def\C{{\Bbb C}} 
\def\Z{{\Bbb Z}} \def\P{{\Bbb P}} \def\R{{\Bbb R}}
\def\cX{{\cal X}}   \def\cY{{\cal Y}}
\def\al{\alpha} \def\be{\beta} \def\ga{\gamma} \def\la{\lambda} 
\def\ep{\varepsilon} \def\de{\delta}

\DeclareMathOperator{\ord}{ord}  
\let\lg\undefined
\DeclareMathOperator{\lg}{lg}  
\DeclareMathOperator{\Sing}{Sing} 
\DeclareMathOperator{\Pic}{Pic}  
\DeclareMathOperator{\Cl}{Cl}

\begin{document}
\title{Non-embeddable $1$-convex manifolds}
\author{Jan Stevens}
\address{Matematiska vetenskaper, G\"oteborgs universitet 
och Chalmers tekniska h\"og\-skola, 41296
G\"oteborg, Sweden}
\email{stevens@chalmers.se}


\begin{abstract}
We show that every small resolution of a three-dimensional
terminal hypersurface singularity can occur on a 
non-embeddable $1$-convex manifold.

We give an explicit example of a non-embeddable manifold
containing an irreducible exceptional rational curve 
with normal bundle of type $(1,-3)$. To this end we
study small resolutions of $cD_4$-singularities.
\end{abstract}

\maketitle
\section*{Introduction}

A $1$-convex (or strongly pseudoconvex) complex manifold $X$
with $1$-dimensional  exceptional set can be  embedded
in some $\C^M\times\P^N$, except possibly when
$\dim X =3$ and  an irreducible component of the exceptional curve
is a rational curve with 
normal bundle of type $(-1,-1)$, $(0,-2)$ or $(1,-3)$.
Non-embeddable examples are known in the first two cases \cite{vvt1,co,BL}.
In this paper we show that last type also occurs.

An irreducible exceptional rational curve $C$  on a 3-dimensional manifold $X$
with normal bundle of type $(a,b)$ with $a+b=-2$ 
blows down to a terminal Gorenstein singularity,
that is,  a \textit{cDV}-singularity. 
This means that the general hyperplane section
through the singular point is Du Val, or in other terminology, a
rational double point.
The simplest possibility is an ordinary double point 
(a 3-dimensional $A_1$-singularity). 
The first example of a non-embeddable $1$-convex 
manifold \cite{vvt1,co}
is a variant of Moishezon's example of a non-projective Moishezon
manifold \cite{mo}. Let $Y\subset  \C^4$ be a general hypersurface of degree 
$d\geq6$ with one $A_1$-singularity and 
let $\wl Y\subset \P^4$ be its projective
closure. A small resolution of $\wl Y$ is non-projective and a small
resolution of $Y$ is non-embeddable.
The explicit example of \cite{BL} for the case of normal bundle $(-1,-1)$
is also of this form. The examples for $(0,-2)$ are similar. They start from 
an equation $f_{2k}$ for a $3$-fold $A_{2k-1}$-singularity, for which a
small resolution with irreducible exceptional set is easily constructed. Let
$f_{2N}$ be a homogeneous polynomial of 
high degree $2N$ with isolated singularity
at $0\in\C^4$.  Then $Y_k=\{f_{2k}+\ep f_{2N}=0\}$ is an affine hypersurface
with non-embeddable small resolution. In \cite{BL} this is shown by 
explicit construction of a $3$-chain with the exceptional set as boundary. 

These examples suggest the following construction. Let $\{f=0\}\subset \C^4$
be a hypersurface with terminal singularity at the origin, admitting a small
resolution. Choose a general enough polynomial $g$ of high enough degree.
A small resolution $X$ of $Y=\{f+\ep g=0\}$ should be a non-embeddable $1$-convex
manifold. Our main result states that this is indeed the case.
The proof uses that $X$ is non-embeddable if and only if the corresponding
small resolution $\wl X$ of the projective closure
$\wl Y$ is non-projective \cite{V} 
(as $g$ is general, the 
hyperplane section $\wl Y_\infty=\{g=0\}\subset \P^3$ is smooth). This follows once
the group of Weil divisors modulo algebraic equivalence has rank one
\cite{ko}.
We show the stronger result that the class group $\Cl(\wl Y)$ is infinite
cyclic: by the Grothendieck-Lefschetz theorem of \cite{RS1} $\Cl(\wl Y)$
injects into  the class group of   $\wl Y_\infty$, and by the classical
Noether-Lefschetz theorem this smooth surface has Picard group $\Z$ for
very general $g$ (meaning for $g$ outside a countable union of subvarieties
in parameter space).

We also provide an explicit example of a non-embeddable $X$ with irreducible
exceptional curve with normal bundle $(1,-3)$. Such a curve blows down to 
a singularity with general hyperplane section of type $D_4$, $E_6$, $E_7$
or $E_8$ \cite{KM}. In the latter cases the formulas become very complicated,
so we restrict ourselves to the simplest one ($D_4$). The strict transform
of the general hyperplane section is  a partial resolution of the $D_4$
singularity, and the total space is a 1-parameter smoothing. It can be obtained
by pull-back from the versal deformation of the partial resolution, which is a
simultaneous partial resolution of the versal deformation of the singularity, 
after a base change. We compute this base change and then construct the
small modification, generalising Example (5.15) in Reid's pagoda paper \cite{re}.
We classify all $1$-parameter smoothings, that is, all $3$-dimensional
singularities to which these blow down.

Our explicit example is a small resolution of 
\[
 x^2+(t+z)y^2+(t-z)z^2-(t^2-z^2)t^{2k}+\ep t^{2m}=0\;.
\]
This hypersurface is very singular at infinity, but the equation has the advantage 
of containing only a few terms.
We explicitly show that
(twice) the  exceptional curve bounds
a real 3-chain, and therefore the small resolution is not embeddable.

In the first  section we recall the necessary definitions 
and known results about 
non-embeddable $1$-convex manifolds. Then we show our main 
result on the existence of hypersurfaces with non-embeddable
small resolution. In the second section we classify the
$cD_4$-singularities which admit a small resolution,
and construct this resolution explicitly. The final section
is devoted to the specific example.


\section{Non-embeddable 1-convex manifolds}
\begin{defn}
A complex space $X$ is \textit{$1$-convex} (or \textit{strongly pseudoconvex}) 
if there 
exists a proper surjective morphism $\pi\colon X \to Y$ onto a Stein 
space $Y$ with $\pi_*\sier X=\sier Y$ and a finite subset $T\subset Y$
such that $X\setminus \pi^{-1} (T) \to Y\setminus T$ is biholomorphic. The 
\textit{exceptional set} is $S=\pi^{-1} (T)$.
\end{defn}

\begin{defn}
A $1$-convex space $X$ is called \textit{embeddable} 
if there exists a holomorphic embedding
$X\to\C^M\times\P^N$ for some $(M,N)$.
\end{defn}

A necessary and sufficient condition is given by the following result
\cite{sch}:

\begin{prop}
The $1$-convex manifold $X$  with exceptional set $S$
is embeddable if and only if  there exists a line
bundle on $X$ with $L_{|S}$ ample.
\end{prop}

We from now on only consider the case of one-dimensional exceptional
sets $S$. 
There is the following topological criterion.
\begin{theorem}[\cite{AB1}]
Let $X$ be  a $1$-convex manifold  with one-dimensional exceptional
set $S$. Then $X$ is K\"ahler if and only if  $S$ does not contain
an effective curve $C$, whose class in $H_2(X,\Z)$ vanishes.
If moreover $H_2(X,\Z)$ is 
finitely generated, then these conditions are equivalent to the 
fact that $X$ is embeddable.
\end{theorem} 
Similar results are obtained in \cite{vvt1}.
Non-embeddable $1$-convex manifolds are very special.

\begin{theorem}[\cite{co,vvt3}]
If a $1$-convex manifold $X$ with one-dimensional exceptional
set $S$ is not embeddable, then $X$ has dimension three and
$S$ has an irreducible component $C$ with $K_X\cdot C=0$.
\end{theorem}

It follows that $C$ is a rational curve, with normal bundle of type
$\sier C(a)\oplus \sier C (b)$, satisfying $a+b=-2$.
The only possibilities are  $(-1,-1)$, $(0,-2)$ or $(1,-3)$
\cite{la}, see also \cite[Lecture 16]{CKM}.

To describe the singularities of $Y$ we look at the germ of $X$ along
$S$. 
Let $\pi\colon (X,S)\to (Y,p)$ be a small contraction 
(we call the map $\pi$ a small contraction or small
resolution
depending on whether we view $X$ or $Y$ as the primary object)
with $X$ smooth
and $K_X$ $\pi$-trivial, i.e., $K_Y\cdot C=0$ for every irreducible
component of $S$. Then $Y$ is Gorenstein terminal, so a 
\textit{cDV}-singularity.
This means that a hyperplane section through the 
singular point is a rational double point
(a.k.a. DuVal singularity).

\begin{prop}[\cite{re}]
If $(H,p)$ is a generic hyperplane section of the cDV-singularity $(Y,p)$
with small resolution  $\pi\colon (X,S)\to (Y,p)$, then $G:=\pi^*H$ is normal
and the induced map $f\colon (G,S)\to  (Y,p)$ is a partial resolution, 
dominated by the minimal resolution.
\end{prop}

In particular, if $\pi$ contracts only one rational curve, the partial
resolution is obtained by blowing down all exceptional curves on the
minimal resolution $\wt H$ of $H$, except one. The type of $H$
is determined by Koll\'ar's length invariant  \cite[Lecture 16]{CKM}.

\begin{defn}\label{lengthdef}
The \textit{length} $l$ of the small contraction  $\pi\colon (X,C)\to (Y,p)$
with irreducible exceptional curve $C$ is 
\[
l=\lg \sier X/\pi^*\mathfrak m_{Y,p}\;.
\]
\end{defn}
The length equals the multiplicity of the fundamental cycle of $\wt H$
at the strict transform of the exceptional curve $C$.

\begin{prop}[\cite{KM}]
The length of the small contraction $\pi$ determines the type of the
general hyperplane section $H$ and the partial resolution $G\to H$.
\end{prop}

A simple proof is given by Kawamata \cite{ka}.
For length $l=1$ the general hyperplane section is of type $A_1$.
This occurs for normal bundle of type $(-1,-1)$ or $(0,-2)$.
For  $(1,-3)$ the length lies between $2$ and $6$, with for $l=2,3,4$
general hyperplane section $D_4$, $E_6$ and $E_7$. If $l=5,6$, then $H$
has an $E_8$-singularity.

This result suggests how to construct examples of small contractions:
start with a partial resolution of a rational double point with 
irreducible exceptional curve, and take a 1-parameter smoothing of it,
such that the exceptional curve is isolated. As the singularity is rational,
the deformation blows down to a deformation of the rational double point.
Typically this construction leads to an affine hypersurface with 
embeddable small resolution.
Examples of non-embeddable spaces with a singularity of length 1 were given
by Col\c toiu, Vo Van Tan and Bassanelli--Leoni \cite{co, vvt1, BL}, see
also \cite{co2,vvt3}.

For an affine threefold $Y$ with small resolution $\pi\colon X\to Y$
embeddability of $X$ is closely related to projectivity
of the corresponding small resolution of the projective closure of $Y$.
More precisely, we have the following result of \cite{AB2},
which was proved earlier  in the special case of hypersurfaces in \cite{V}.
A similar result is proved in \cite{vvt1}.

\begin{theorem}[\cite{AB2}]
Let  $\pi\colon X\to Y$ be a contraction of the $1$-convex manifold 
$X$, with $Y$ Stein and quasi-projective, of dimension at least
3.
Let $(\wl Y,\wl Y_\infty)$ be the projective closure of $Y$
and assume that $\Sing(\wl Y)=\Sing(Y)$. Let $(\wl X,\wl X_\infty)$ 
be the corresponding compactification, with the same divisor 
$\wl X_\infty=\wl Y_\infty$ at infinity. 
Suppose that the map
$H_2(X,\R)\to H_2(\wl X,\R)$ is injective.
Then $X$ is embeddable (this is so  if and only if $X$ is K\"ahler)
if and only if $\wl X$ is projective.
\end{theorem} 

The condition on the map $H_2(X,\R)\to H_2(\wl X,\R)$ is in particular
satisfied if $\wl Y_\infty$ is a smooth projective hypersurface.

For small resolutions of threefolds in $\P^4$ we have the following
result. 
\begin{theorem}[{\cite[Theorem 5.3.2]{ko}}]
Let  $\pi\colon \wl X\to \wl Y$ be a small resolution of 
a projective threefold $\wl Y$ with at most terminal hypersurface 
singularities, such that the group of Weil divisors modulo
algebraic equivalence has rank one. 
Then $\wl X$ is a Moishezon threefold which is nonprojective
if $\pi$ is not an isomorphism.
\end{theorem} 

Consider now any terminal hypersurface singularity, which admits a
small resolution. Then there exists a projective hypersurface with
this singularity as only singularity.

\begin{lemma}
Let the polynomial function
$f\colon \C^4\to \C$ define a hypersurface 
with isolated singularity at the origin.  
For generic  homogeneous $g$ of high enough degree the projective
closure $\wl Y\subset \P^4$ of the affine hypersurface $V(f+\ep g)$
has only one singularity, isomorphic to the singularity
of $V(f)$  at the origin, and the hyperplane section at infinity 
$\wl Y_\infty$ is smooth.
\end{lemma}
\begin{proof}
As the singularity of $f$ at the origin is finitely determined, the 
hypersurface $V(f+\ep g)$ has an isomorphic singularity at the origin if
the degree of $g$ is at least the degree of determinacy. For generic
$g$ there are no other singularities and  
the hyperplane section at infinity $\wl Y_\infty=V(g)\subset \P^3$
is smooth.
\end{proof}

A polynomial is \textit{very general}, if its parameter point lies
outside a countable union of proper subvarieties in the
space parametrising polynomials of given degree.
\begin{theorem}
Let $f$ and $g$ be as in the lemma above, and let $\wl Y\subset \P^4$
be the projective closure  of $V(f+\ep g)$.
If $g$ is a very general polynomial, then the class group of
$\wl Y$ satisfies $\Cl(\wl Y)\cong \Z$.
\end{theorem} 
\begin{proof}
By the Grothendieck--Lefschetz theorem of Ravindra and Srinivas
\cite{RS1} the restriction homomorphism 
$\Cl(\wl Y) \to \Cl(\wl Y_ \infty)$ is injective. The theorem as stated there
gives only the conclusion for hyperplane sections in a Zariski dense open
subset of sections, but as remarked by the same Authors 
in \cite[p. 3378]{RS2},
it suffices that $\wl Y_ \infty$ is smooth and does not pass through
the singularity of $\wl Y$. 
By the classical Noether--Lefschetz theorem a very general smooth
surface $S$ of degree at least $4$ in 
$\P^3$ satisfies $\Cl(S)=\Pic(S)\cong \Z$. An algebraic 
proof can be found in 
\cite{RS2}. 
\end{proof} 

Combining the above results we obtain that every small contraction 
to a hypersurface singularity can occur on a non-embeddable 
$1$-convex manifold.
\begin{cor}\label{result}
Suppose that the affine threefold $V(f)\subset \C^4$  has a terminal 
hypersurface singularity at the origin, which admits a 
non-trivial small resolution $X_0$. 
For very general   homogeneous $g$ of high enough degree 
the corresponding small resolution $X $  
of the affine hypersurface $V(f+\ep g)$
is a non-embeddable $1$-convex manifold.
\end{cor}

The above Corollary is a  statement about affine 3-folds. A direct proof,
without going to the projective closure, would be preferable.
We have not been able to find it. The problem is that
there exist affine hypersurfaces with terminal singularities,
whose small resolution is embeddable: typically 
this is the case for the hypersurface
$V(f)$, whose small resolution is given by explicit polynomial formulas.
Adding the form $g$ means specifying the hyperplane section
at infinity, so one is naturally led to  the projective closure. 

\section{Small resolutions for $cD_4 $-singularities}
In this section we construct a small resolution
with irreducible exceptional set for certain 
$cD_4$-singularities.
We view it as total space of a 1-parameter smoothing of a 
partial resolution of a $D_4$ surface singularity.
As such it can be obtained by pull-back from the 
versal deformation of the partial resolution.
We first describe  Pinkham's construction 
of this versal deformation \cite{pi},
see also \cite{KM}. We give explicit formulas.
We then classify the occurring singularities.

The versal deformation $\cY \to S$  of a surface singularity $Y$
of type $A$, $D$, $E$ 
admits a simultaneous resolution after base change with the
corresponding Weyl group $W$.
We write $S=T/W$ and identify
$T$ with the vector space spanned by a root system of 
type $A$, $D$ or $E$. 
The simultaneous resolution is 
the versal deformation $\wt \cX \to T$ of the 
minimal resolution $\wt X$ of $Y$.

Now consider a partial resolution $\wh X \to Y$
with irreducible exceptional set $E_0$; 
we denote  strict transform of $E_0$ on the minimal
resolution by the same name. 
It determines a one vertex subgraph $\Gamma_0$
of the resolution graph 
$\Gamma$. The connected components of the complement 
$\Gamma\setminus\Gamma_0$ are the graphs of the singularities
on the partial resolution $\wh X$; we can construct
$\wh X$
from the minimal resolution by blowing down the configurations
of curves, given by $\Gamma\setminus\Gamma_0$. 
The versal
deformation $\wh\cX$ of $\wh X$ 
admits a simultaneous resolution after base 
change with the product $W_0$ of the Weyl groups corresponding to
the connected components of  
$\Gamma\setminus\Gamma_0$; this simultaneous
resolution is nothing else than $\wt\cX\to T$. 
So the base space
of $\wh\cX$ is $T/W_0$.

In the cases $A$ and $D$ it is rather easy to give the
simultaneous partial resolution explicitly, but for $E$ (and
especially $E_8$) the formulas become too complicated (cf.~\cite{KM}).
We restrict ourselves in the following to the simplest case
leading to normal bundle $(1,-3)$, the case $D_4$.
Then the length of the small contraction (Definition 
\ref{lengthdef}) is equal to two.

We start from the versal deformation $\cX \to S=T/W$ 
of $D_4$, 
as in  \cite{tj}, see also \cite{KM}, given by
\begin{equation}\label{versald4}
  x^2+y^2z-z^3-t_2z^2-t_4z-t_6+2s_4y=0\;.
\end{equation}
In $T\cong \C^4$, with negative definite inner product 
$\lbrace e_i, e_j \rbrace = -\delta_{ij}$,
 there is a root system with basis 
\begin{alignat*}{2}
v_1 &= e_1 - e_2 &&= (1,-1,0,0)\\
v_2 &= e_2 - e_3 &&= (0,1,-1,0)\\
v_3 &= e_3 - e_4 &&= (0,0,1,-1)\\
v_4 &= e_3 + e_4 &&= (0,0,1,1)
\end{alignat*}
and Dynkin diagram
\[
\unitlength 30pt
\begin{picture}(2,1)
\put(0,0){\makebox(0,0){$v_1$}}
\put(0.2,0){\line(1,0){0.6}}
\put(1,0){\makebox(0,0){$v_2$}}
\put(1.2,0){\line(1,0){0.6}}
\put(2,0){\makebox(0,0){$v_3$}}
\put(1,0.2){\line(0,1){0.6}}
\put(1,1){\makebox(0,0){$v_4$}}
\end{picture}
\]
The roots $v_i$ correspond to the components of the exceptional divisor
of the resolution of $D_4$. 
The partial resolution $\wt X$, which only pulls out the 
central curve, is obtained by blowing down the components corresponding
to $v_1$, $v_3$ and $v_4$. 
The graph $\Gamma\setminus\Gamma_0$ has three components,
all three of type $A_1$. 
Its Weyl group is $W_0=W(A_1)\times W(A_1)\times W(A_1)$.
To describe it explicitly, we take coordinates  
$(\al_1,\al_2,\al_3,\al_4)$ on $T$.
The reflection  $s_{v_4}$ acts as 
$(\al_1,\al_2,\al_3,\al_4)
\mapsto(\al_1,\al_2,-\al_4,-\al_3)$, whereas
for $i=1,2,3$ the $s_{v_i}$ are the transpositions $(i,i+1)$.
The connection with the deformation \eqref{versald4}
is that the coordinates on $S=T/W$ are the invariants
$t_{2i}=\sigma_i(\al_1^2,\al_2^2,\al_3^2,\al_4^2)$  for $i=1,2,3$ and 
$s_4=\sigma_4(\al_1,\al_2,\al_3,\al_4)$, where the $\sigma_i$ are the 
elementary symmetric functions.

\begin{prop}
The versal   deformation $\wh \cX \to T/W_0$ of  the
partial resolution 
$\wh X$ is a simultaneous partial resolution (without base 
change) of the deformation
\begin{multline}\label{versal}
F(x,y,z;\be_1,\be_2,\ga_3,\be_4)={}\\
  x^2-(z^2+z\be_1+\be_2^2)\ga_3^2
 +z(y+\be_2)^2-2\be_2(y+\be_2)(z-\be_4)-(z+\be_1)(z-\be_4)^2
\end{multline}
of the  $D_4$ surface singularity $Y$.
\end{prop}

\begin{proof}
The invariants for $W_0$ are
\begin{align*}
\ga_3&=\al_1+\al_2\;,\\
\be_1&=\al_3^2+\al_4^2\;,\\
\be_2&=\al_3\al_4\;,\\
\be_4&=\al_1\al_2\;.
\end{align*}
We express the coordinates on $S$ in these invariants:
\begin{align*}
t_2&=\be_1+\ga_3^2-2\be_4\;,\\
t_4&=\be_2^2+\be_4^2+\be_1(\ga_3^2-2\be_4)\;,\\
t_6&=\be_1\be_4^2+\be_2^2(\ga_3^2-2\be_4)\;,\\
s_4&=\be_2\be_4\;.
\end{align*}
%
Inserting these values in the versal family 
\eqref{versald4}
and rearranging gives the formula \eqref{versal}.
According to Pinkham \cite{pi} 
a simultaneous partial resolution gives the desired versal
deformation.
\end{proof}

\begin{lemma}
The (reduced) discriminant of the family \eqref{versal}
has five irreducible components, given by 
$4\be_4=\ga_3^2$, $\be_2=2\be_1$, $\be_2=-2\be_1$, $\ga_3=0$
and
\begin{equation}\label{discomp}
(\be_4^2+\be_1\be_4+\be_2^2)^2-\ga_3^2(\be_1\be_2^2+4\be_2^2\be_4+\be_1\be_4^2)
+\be_2^2\ga_3^4=0\;.
\end{equation}
\end{lemma}
\begin{proof}
The discriminant is the image of the reflection
hyperplanes  $\al_i\pm\al_j=0$ in $T$.
The hyperplanes perpendicular to $v_1$, $v_3$ and $v_4$ are $\al_1=\al_2$,
$\al_3=\al_4$ and $\al_3=-\al_4$; they map to $4\be_4=\ga_3^2$, $\be_2=2\be_1$
and $\be_2=-2\be_1$. The fundamental cycle of the singularity corresponds
to $v_1+2v_2+v_3+v_4=(1,1,0,0)$ and determines the hyperplane
$\al_1+\al_2=\ga_3=0$. For $\ga_3=0$ there is a singular point at 
$x=y+\be_2=z-\be_4=0$.

The remaining hyperplanes give rise to the same 
irreducible component of the discriminant.
To describe it we determine the corresponding component of the critical locus,
which is the component of the singular locus of the total space, not contained in
$\ga_3=0$. A computation, which we suppress,
shows that it  is given by $x=0$ and
\[
\text{Rank}
\begin{pmatrix}
z & \be_2 & z-\be_4  \\
\be_2 & -z-\be_1 & y+\be_2  \\
z-\be_4 & y+\be_2 & -\ga_3^2
\end{pmatrix}
\leq 1\;.
\]
This is indeed image of $\al_1+\al_3=0$:
we have $\ga_3=\al_1+\al_2$,
$\be_1=\al_1^2+\al_4^2$, $\be_2=-\al_1\al_4$ 
and $\be_4=\al_1\al_2$, while the singular point lies
at $x=0$, 
$y=-\al_2\al_4$ and $z=-\al_1^2$.
By  eliminating the variables $y$ and $z$
we find the equation \eqref{discomp} for this component.
\end{proof}

To explicitly construct  the simultaneous
partial resolution we proceed as in  
Example (5.15) of \cite{re}.
We write the family \eqref{versal} as
\begin{equation}\label{pagoda}
F=X^2+(ac-b^2)T^2+aY^2-2bYZ+cZ^2\;.
\end{equation}
This generalises the family in \cite{re}, where $b=0$.
The coordinate change is given by $X=x$, $T=\ga_3$, $Y=y+\be_2$, $Z=z-\be_4$,
$a=z$, $b=\be_2$ and $c=-z-\be_1$.
We consider the family \eqref{pagoda}
as quadric in $X$, $Y$, $Z$ and $T$, with coefficients
in $\C[a,b,c]$. It has two small resolutions. 
To give them explicitly we have to
factorize $aY^2-2bYZ+cZ^2$. The idea is to put $a=-\al^2$ and write
$X^2-(\al Y+\frac b{\al}Z)^2+(ac-b^2)(T^2-(\frac Z{\al})^2)=0$. Then one small
resolution is obtained by blowing up the ideal $(X-(\al Y+\frac b{\al}Z),T-\frac Z{\al})$.
We could as well set $c=-\ga^2$; it leads to the same two small resolutions.

We eliminate $\al$ from the generators of the ideal by writing them as
$(-1,\al)M$, where $M$ is the matrix
\begin{equation}
M=
\begin{pmatrix}
Z  &  -aT  & X-bT  & -aY \\
T  &    Z   &   Y   &  X-bT
\end{pmatrix}\;.
\end{equation}
\begin{lemma}
The blow-up of the ideal generated by the minors of the $2\times4$ matrix
$M$  defines a small resolution $\wh \cX$ of the total space
of the family \eqref{pagoda}.
\end{lemma}

\begin{proof}
The minors of the matrix $M$ are not independent, and the ideal needs only
four generators. 
The blow-up is the subset of $\C^3\times\C^4\times\P^3$,
which is the closure of the graph of the rational map
\[
(P:Q:R:S)=
(Z^2+aT^2 :  (X-bT)^2+aY^2  : ZY-T(X-bT)  : Z(X-bT)+aTY)\;.
\]
The relation between the minors gives $S^2-PQ+aR^2=0$.
Furthermore one sees that $Q+cP-2bR$ is proportional to $F$, so vanishes
on $F=0$. This  allows us  to eliminate $Q=2bR-cP$.
We find
\[
S^2+cP^2-2bPR+aR^2=0\;.
\]
This formula shows that we have a small modification.
The determinantal syzygies between the minors
of the matrix give the following four equations, where $Q$ is already
eliminated:
\[
\begin{array}{c@{}c@{}r}
YP-ZR-TS &{}={}&0\;, \\
(X-bT)P+aTR-ZS &=&0\;, \\
-cTP+(X+bT)R-YS &=&0\;, \\
(cZ-bY)P+(aY-bZ)R+XS &=&0\;. 
\end{array}
\]
These equations determine $(P:R:S)$ except when the rank of the coefficient
matrix is at most one: 
this happens exactly at the singular points.

To check smoothness we look at affine charts. The exceptional curve over the
origin is covered by  the charts $P=1$ and $R=1$.
In $R=1$ we can eliminate $Z$, $X$ and $a$, leaving
$(Y,P,S,T,b,c)$ as coordinates:
\begin{align*}
Z&=YP-TS\;,\\
X+bT&=cTP+YS\;,\\
   S^2+cP^2-2bP+a&=0\;.
\end{align*}
This shows that the space $\wh \cX$ is smooth in this chart.
Likewise we find in the chart $P=1$ that
\begin{align*}
Y&=ZR+TS\;,\\
X-bT&=-aTR+ZS\;,\\
S^2+c-2bR+aR^2&=0 \;.
\end{align*}
\end{proof}

To view $\wh \cX$ as simultaneous partial resolution
of the deformation \eqref{versal} of $D_4$ we have to go
back to the original coordinates.
The three singular points of the special fibre are
visible in the chart $R=1$, so we
only look at this chart. We eliminate $x$ and $z$
via
\begin{equation}\label{formules}
\begin{aligned}
x&=(y+\be_2)S-(z+\be_1)\ga_3P-\be_2\ga_3\;,\\
z&=(y+\be_2)P-\ga_3S+\be_4
\end{aligned}
\end{equation}
and are left with one equation
\begin{equation}\label{preseq}
S^2-(y+\be_2)P^3-(\be_1+\be_4-\ga_3S)P^2+(y-\be_2)P
+\be_4-\ga_3S=0   
\end{equation}
in the variables $(S,P,y;\be_1,\be_2,\ga_3,\be_4)$. 

Over the non-smooth component \eqref{discomp}
of the discriminant we have, using the parametrisation by the
reflection hyperplane $\al_1+\al_3=0$, that
\begin{align*}
x&=(\al_1+\al_2)(-\al_4S+(\al_1+\al_2)P)(S+\al_4P-\al_1)\;,\\
z&=-\al_1^2-(\al_1+\al_2)(S+\al_4P-\al_1)\;,\\
0&=(S+\al_4P-\al_1)(S+(\al_1+\al_2)P^2-\al_4P-\al_2)\;.
\end{align*}
So the curve $S+\al_4P-\al_1=0$ is the exceptional curve 
(it extends to the $P=1$
chart).

Over $\ga_3=0$ we have $x=y+\be_2=z-\be_4=0$ and the exceptional curve
$S^2-(\be_1+\be_4)P^2-2\be_2P+\be_4=0$.
Note that this curve is reducible if in addition $\be_2^2+\be_1\be_4+\be_4^2=0$,
that is, over the intersection of the two components of the discriminant. 

We now return to $3$-dimensional $cD_4$ singularities. We use
Arnol'd's notation for singularities, see \cite{AGV}.
\begin{prop}
A $3$-fold singularity with $D_4$ as general hyperplane section,
which has a small resolution with irreducible exceptional curve,
is of type $T_{3,3,2q+2}$, $Q_{6q+5}$ or $Q_{\rho+1,\de}$ with $\de$ odd.
\end{prop}

\begin{proof}
The singularity is a 
1-parameter smoothing of the hyperplane section $D_4$, 
so can be obtained by pull-back from the versal family \eqref{versal}.
We have to describe a curve 
in the base space, so now we take the $\be_i$ and $\ga_3$ to be
functions of a variable $t$. Having a small resolution
with irreducible exceptional curve
gives two conditions, that the curve  does not lie in the discriminant, 
and that the total space of the  1-parameter deformation
of the partial resolution of $D_4$ is smooth.
The first condition translates into $\ga_3(t)\not\equiv0$ and a more complicated
one for the other component.  Smoothness of the total space can be checked
in the $R=1$ chart. We  look at the equation 
\eqref{preseq}
and its derivatives with respect to the variables $y$, $S$,
$P$  and $t$. A possible singular point satisfies
\[
\begin{array}{c@{}c@{}r}
(P^2-1)P&=0\;,\\
2S+\ga_3(P^2-1)&=0\;,\\
-3P^2(y+\be_2)-2(\be_1+\be_4-\ga_3s)P+y-\be_2&=0\;,\\
-\be_1'P^2-\be_2'(P^3+P)+(\ga_3's-\be_4')(P^2-1)&=0\;.
\end{array}
\]
Here $\be_i'(t)$ is the derivative of the power series $\be_i(t)$.
If $P=0$, $y=\be_2$ and $2S=\ga_3$,
$S^2+\be_4-\ga_3S=0$, so $4\be_4=\ga_3^2$. The condition of nonsingularity
is then that $2\ga_3\ga_3'-4\be_4'\neq0$, at $t=0$.
If $P=\pm1$, $S=0$, $\be_1\pm2\be_2=0$ and $y\pm\be_4=0$
and we get the condition at  $t=0$  that 
$\be_1'\pm2\be_2'\neq0$.

As $\ga_3(0)=0$ the condition at  $P=0$ becomes $\be_4'\neq0$, so
by a coordinate change we may assume that $\be_4(t)=t$.
We now write $\be_i(t)=t\bar \be_i(t)=t(b_1+\dots)$ and 
$\ga_3(t)=t\bar\ga_3(t)=t(c_3+\dots)$.
We put $F=x^2+G(y,z,t)$.
The $3$-jet of $G$ is
\[
j^3G=z(y+b_2t)^2-2b_2t(y+b_2t)(z-t)+(-z-b_1t)(z-t)^2\;.
\]
This cubic defines a cubic curve in $\P^2$ with singular point in 
$(y:z:t)=(-b_2:1:1)$. We claim that it is irreducible. To show this we
compute a parametrisation using the pencil
$ \la (y+b_2t)=\mu(z-t)$ of lines through the singular point.
The curve is irreducible if and only if it has a rational parametrisation
of degree 3.
We find 
\[
(z-t)^2((\mu^2-\la^2)z-(2\mu b_2+b_1\la))\la t=0\;,
\]
so the curve is reducible if and only if   $\mu^2-\la^2$ and
$(2\mu b_2+b_1\la)\la$ have a factor in common, but this only happens
if $b_1\pm2b_2=0$, which is excluded by  non-singularity.
The cubic has a cusp if $1+b_1+b_2^2=0$.
Otherwise the curve is a nodal cubic, and the singularity is a cusp
of type $T_{3,3,r}$.  To determine the exact type of the singularity, also
in the cuspidal cubic case, we blow up the origin. As said, we write
$\be_i=t\bar\be_i$, $\ga_3=t\bar\ga_3$.  We look at the appropriate chart,
with coordinates $(t,\eta,\zeta)$, such that $(t,y,z)=(t,\eta t, \zeta t)$.

The strict transform of $G$ is
\[
(-\zeta^2-\zeta\bar\be_1-\bar\be_2^2)t\bar\ga_3^2
  +\zeta(\eta+\bar\be_2)^2-2\bar\be_2(\eta+\bar\be_2)
  (\zeta-1)+(-\zeta-\bar\be_1)(\zeta-1)^2\;.
\]
The singular point lies at $(t,\eta,\zeta)=(0,-b_2,1)$. 
We  multiply the equation with the unit $\zeta$ 
and complete the square to obtain
\[
(\zeta\eta+\bar\be_2)^2-
(\zeta^2+\zeta\bar\be_1+\bar\be_2^2)(\zeta t\bar\ga_3^2+(\zeta-1)^2)\;.
\]

If $1+b_1+b_2^2\neq0$, then $\zeta^2+\zeta\bar\be_1+\bar\be_2^2$ is a
unit, and the singularity on the strict transform is an $A_{2q-2}$ with
$q=\ord_t\ga_3=\ord_t\bar\ga_3+1$ and the original singularity is of type
$T_{3,3,2q+2}$.

Otherwise $\zeta^2+\zeta\bar\be_1+\bar\be_2^2$ is the equation 
of a curve, which
is smooth in the point
$(t,\zeta)=(0,1)$, as $b_1=-2$ and $1+b_1+b_2^2=0$ gives $b_2^2=1$,
contradicting the condition $b_1\neq\pm2b_2$ for nonsingularity.
Let $\rho=\ord_t(1+\bar\be_1+\bar\be_2^2)$. The order of
contact of the smooth branch with the cusp 
$\zeta t\bar\ga_3^2+(\zeta-1)^2$ is equal to $\min(2q-1,2\rho)$. 
If the minimum is $2q-1$, then there is an $E_{6q-5}$ and the
original singularity is of type $Q_{6q+5}$. Otherwise we set
$\delta=2(q-\rho)-1$; the singularity is of type $J_{\rho,\delta}$
with  original singularity of type $Q_{\rho+1,\delta}$.
\end{proof}

\begin{remark}
The original example of Laufer \cite{la} of an exceptional curve
with normal bundle of type $(1,-3)$ is 
$x^2+y^3+zt^2+yz^{2q+1}$ of type $Q_{6q+5}$. 
One needs a coordinate transformation to bring this
equation into our normal form. Note that the general 
hyperplane section does not give the standard 
quasi-homogeneous form for $D_4$.
\end{remark}

\section{A specific example}
We now give an example of a non-embeddable
$1$-convex manifold.
To have one we can compute with, we look for a 
simple formula with only a few terms.

In the versal family \eqref{versal} we substitute
$\be_2=0$, $\be_4=t$, $\be_1=-2t$ and $\ga_3=it^k$.
After the coordinate transformation $z\mapsto z+t$
we obtain
the $3$-fold singularity 
\[
f=x^2+(t+z)y^2+(t-z)z^2-(t^2-z^2)t^{2k}
\]
of type $T_{3,3,2k+2}$. 
The small resolution of the previous section 
gives an embeddable $1$-convex manifold.
The given formula determines a curve in the base 
of the versal deformation \eqref{versal}, which
intersects the discriminant in $t^4-2t^{2k+3}=0$, 
so the hypersurface
$\{f=0\}\subset\C^4$ has singular points
for $t^{2k-1}=\frac12$. 
They are also resolved by the construction.

Now we perturb the function $f$ by adding terms  of 
high order. 
We take only one monomial, which makes the resulting
hypersurface  very singular
at infinity. It is therefore not an example for
Corollary \ref{result}. We show that a small
resolution is non-embeddable by explicitly exhibiting
a $3$-chain with boundary on the exceptional curve.

\begin{prop}
\label{onesing}
The affine hypersurface with equation
\[
h= x^2+(t+z)y^2+(t-z)z^2-(t^2-z^2)t^{2k}+\ep t^{2m}=0\;,
\]
where $m>k+1$,
has for almost all $\ep$ only one singular point, 
of type $T_{3,3,2k+2}$, isomorphic to that of $f$.
\end{prop}

\begin{proof}
We compute the singular locus 
$V(h,\partial_x h, \partial_y h,\partial_z h,\partial_th)$:
\begin{align*}
h &= x^2+(t+z)y^2+(t-z)z^2-(t^2-z^2)t^{2k}+\ep t^{2m}\;,\\
\partial_x h &= 2x\;,\\
\partial_y h&= 2(t+z)y\;,\\
\partial_z h&= y^2+2tz-3z^2+2zt^{2k}\;,\\
\partial_t h&= y^2+z^2-(2k+2)t^{2k+1}+2kz^2t^{2k-1}+2m\ep t^{2m-1}\;.
\end{align*}
A singular point always satisfies $x=0$.
If $z+t=0$, then $y^2=5t^2+2t^{2k+1}$ and $2t^3+\ep t^{2m}=0$.
This makes that $t\partial_th=6t^3+2m\ep t^{2m}=
(6-4m)t^3$, so $t=z=y=x=0$. At the origin $h$ has the
same singularity as $f$ (a  cusp singularity has no moduli).
\\
If $z+t\neq 0$, then $y=0$. If also  $z=0$ holds, 
then $t^{2k+2}=\ep t^{2m}$, and
$t\partial_th=2(m-k-1)t^{2k+2}\neq 0$, as $t\neq 0$. 
\\
If $z\neq0$, then $3z=2t+2t^{2k}$. This shows that 
$t\neq0$. Therefore
\[
\begin{array}{c@{}c@{}r}
(t-2t^{2k})^2(4t+t^{2k})+27\ep t^{2m}&=0\;,\\
(t-2t^{2k})(4t-2t^{2k}
-2k(5t+2t^{2k})t^{2k-1})+18m\ep t^{2m-1}&=0\;.
\end{array}
\]
The first equation shows that $t^{2k-1}=\frac12$ gives no longer
singular points for $\ep\neq 0$
We eliminate $\ep$, divide by $2t^2(t-2t^{2k})$ and find
\[
(4m-6)+(15k+3-7m)t^{2k-1}+(6k-2m)t^{4k-2}=0\;.
\]
This is a quadratic equation for $t^{2k-1}$. 
Only for finitely many values of $\ep$ there are singular
points outside the origin. If $m\neq3k$, then
$3(m-3k)^2\ep t^{2m-3}+(4m-10k-1)(2k-1)t^{2k-1}-2(2mk-m-2k^2-k+1)=0$. 
For $m=3k$ the equations simplify:
$t^{2k-1}=2$,  $\ep t^{2m-3}=-2$, but on the other hand
$t^{2m-3}=t^{6k-3}=8$, so singularities only 
exist for $\ep=-1/4$.
\end{proof}


\begin{theorem}
A small resolution of the affine hypersurface $\{h=0\}$,
with $h$ as in Proposition \ref{onesing} and $\ep >0$, 
is a non-embeddable 1-convex manifold,
with rational irreducible exceptional curve with normal
bundle of type $(1,-3)$. 
\end{theorem}

\begin{proof}
The normal bundle on a small resolution
is as stated, because the general hyperplane
section through the singular point is of type $D_4$.

We prove that the manifold is not embeddable by showing that the exceptional curve
$C$ is  rationally zero-homologous: $2C$ is a boundary.

We write $h(x,y,z,t)=x^2+\bar h_t(y,z)$ and $f(x,y,z,t)=x^2+\bar f_t(y,z)$,
and consider $\bar h_t(y,z)$  and $\bar f_t(y,z)$ as a families 
of affine cubic curves.
For all real $t>0$ the curve $\bar f_t(y,z)$
has three infinite branches and an oval with the
origin in its interior, except for the $t$-value $t^{2k-1}=\frac12$, 
when the total space has a singular point. Then the intersection
with the $z$-axis, given by $(t-z)(z^2-(t+z)t^{2k})=0$,
has $z=t$ as double root.
We obtain $\bar h_t(y,z)$  by adding  the term $\ep t^{2m}$ to
$\bar f_t(y,z)$. As $\ep>0$, there is no longer a double root.
With increasing $t$ the oval becomes smaller, and vanishes if
$t^{2k+2}=\ep t^{2m}$. This equation has only one
real solution. For that $t$-value the curve $\bar h_t(y,z)$
has a singularity, as is  easily seen from the computations in the proof of 
Proposition \ref{onesing}; this singularity is an isolated real point.
We show the curves for $\ep=1$, $k=2$ and $m=6$.
The pictures are made with the {\sc Xalci} web demo
at {\tt exacus.mpi-inf.mpg.de}.
Figure \ref{cubicsfig} shows that for small $t$ the curves 
$\bar h_t$ and $\bar f_t$ look almost the same, and that $\bar h_t$
does not have a double point.
Figure \ref{vanishfig} shows how the oval of $\bar h_t$ 
first grows and then vanishes.
\begin{figure}
\includegraphics[width=.7\textwidth]{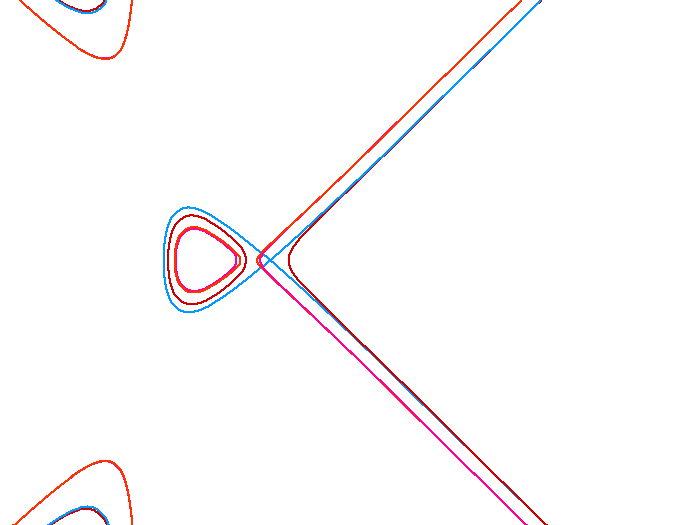}
\caption{$\bar h_t$ and $\bar f_t$ with
$k=2$, $m=6$ and $\ep=1$ for $t=\frac 23, \sqrt[3]{\frac12}$.}
\label{cubicsfig}
\end{figure} 
\begin{figure}
\includegraphics[width=.7\textwidth]{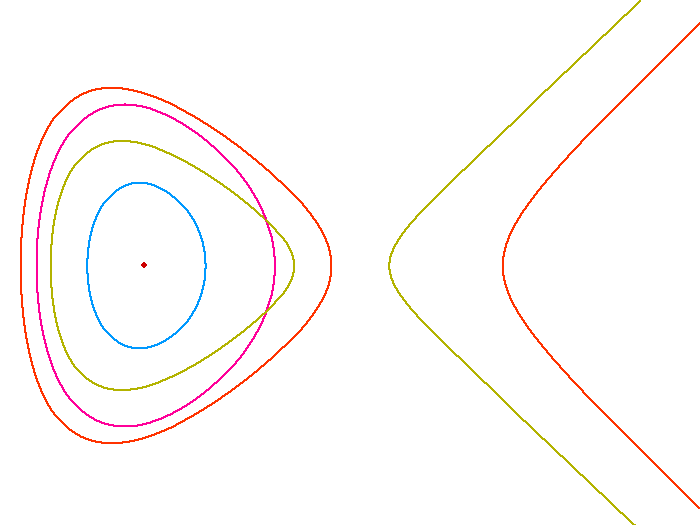} 
\caption{$\bar h_t$  with
$k=2$, $m=6$ and $\ep=1$ for $t=\frac 23, \sqrt[3]{\frac12},
\frac{19}{20},\frac{99}{100},1$.}
\label{vanishfig}
\end{figure}
The family of surfaces $x^2+\bar h_t(y,z)$ 
is the family of double covers 
of the $(y,z)$-plane, branched along the curves
$\bar h_t(y,z)$. For 
$0<t< \sqrt[2m-2k-2]{1/\ep}$ there is a component of the real locus,
which is a double covering of the interior of the  oval, 
branched along the oval itself, while for
$t= \sqrt[2m-2k-2]{1/\ep}$ there is an isolated real point.
The component is
diffeomorphic to a 2-sphere. Together with the isolated point 
they form a smooth
real $3$-dimensional manifold 
$M$ in the  half-space $\{t >  0\}$, which is
compactified by the singular point at the origin.

On the small resolution the manifold $M$ has boundary on the 
exceptional set. 
To compute it, we look at $f$.
For small $t>0$ the  value of $z^2$ on the oval is approximately at most
$t^{2k+1}$, so $|z|\ll t$.  We divide the equation by the unit
$t^2-z^2$ and use
the coordinate transformation 
\[
\xi=\frac{x}{\sqrt{t^2-z^2}},\qquad 
 \eta=\frac{y}{\sqrt{t-z}},\qquad \zeta=\frac{z}{\sqrt{t+z}},
\] 
valid in  a 
neighbourhood of the oval.
Now
\[
\frac{x^2}{t^2-z^2}=\xi^2, \qquad
 \frac{y^2}{t-z}=\eta^2,\qquad \frac{z^2}{t+z}=\zeta^2,
\]
so the transformation brings the 2-sphere in evidence:
\[
\xi^2+\eta^2+\zeta^2=t^{2k}\;.
\]
We have to compute the limit for $t\to 0$ on the small resolution.
We look at the chart $R=1$. 
Rather than computing the inhomogeneous coordinates
$P$ and $S$ from the homogeneous expressions $P/R$ and $S/R$, we find 
$P$  and $S$ from the formulas \eqref{formules} for $z$ and  $x$, 
which after our substitution and
coordinate transformation $z\mapsto z+t$ become
\begin{align*}
x&=yS-i(z-t)t^kP\;,\\
z&=yP-it^kS\;.
\end{align*}
This gives us 
\[
S=\frac{yx+iz(z-t)t^{k}}{y^2-(t-z)t^{2k}}\;,
\qquad
P=\frac{yz+ixt^{k}}{y^2-(t-z)t^{2k}}\;.
\]
In the coordinates introduced above
\[
S=\frac{\eta\xi-i\zeta t^{k}}{\eta^2-t^{2k}}\sqrt{t+z}\;,
\qquad
P=\frac{\eta\zeta+i\xi t^{k}}{\eta^2-t^{2k}}\sqrt{\frac{t+z}{t-z}}\;.
\]
We do not express the square roots in the variable $\zeta$, but observe
that on our component of the real locus
\[
\lim_{t\to0}\sqrt{t+z}=0\;,\qquad
\lim_{t\to0}\sqrt{\frac{t+z}{t-z}}=1\;.
\]
We parametrise the $2$-sphere of radius $t^k$ 
with the inverse of a stereographic projection:
with $w=u+iv$ we put 
\begin{align*}
 \xi&=\frac{2v}{w\bar w+1}t^{k}\;,\\
\eta&=\frac{w\bar w-1}{w\bar w+1}t^{k}\;,\\
\zeta&=\frac{2u}{w\bar w+1}t^{k}\;.
\end{align*}
With these values we find
\[
\lim_{t\to0}S=0\;,
\qquad
\lim_{t\to0}P=\frac{2(w\bar w-1)u+2i(w\bar w+1)v}{-4w\bar w}=
\frac12\left(\frac1w-w\right)\;.
\]
The map $P=(w^{-1}-w)/2$ is degree 2 map from $\P^1$ to $\P^1$, 
showing that the
boundary of the real  manifold $M$  is $2C$. 
\end{proof}

\end{document}